\documentclass[11pt]{article}
\usepackage{eurosym}
\usepackage{amsfonts}
\usepackage{graphicx}
\usepackage{amsmath}

\numberwithin{equation}{section}

\newtheorem{theorem}{Theorem}[section]

\newtheorem{condition}[theorem]{Condition}

\newtheorem{corollary}[theorem]{Corollary}

\newtheorem{lemma}[theorem]{Lemma}

\newtheorem{proposition}[theorem]{Proposition}
\newtheorem{remark}[theorem]{Remark}

\newenvironment{proof}[1][Proof]{\textbf{#1.} }{\ \rule{0.5em}{0.5em}}

\begin{document}

\title{BMO Martingales and Positive Solutions of Heat Equations}
\author{Ying Hu\thanks{%
IRMAR, Universit\'e Rennes 1, 35042 Rennes Cedex, France. Email:
ying.hu@univ-rennes1.fr. This author is partially supported by the Marie
Curie ITN Grant, ``Controlled Systems'', GA no.213841/2008.} \ and Zhongmin
Qian\thanks{%
Mathematical Institute, University of Oxford, Oxford OX2 6GG, England.
Email: qianz@maths.ox.ac.uk.}}
\maketitle

\begin{abstract}

In this paper, we develop a new approach to establish gradient estimates for positive solutions to the heat equation of elliptic or subelliptic operators on Euclidean spaces or on Riemannian manifolds.
More precisely, we give some  estimates
of the gradient of logarithm of a positive solution via the uniform bound of
the logarithm of the solution. Moreover, we give a generalized version of
Li-Yau's estimate. Our proof is based on the link between PDE and quadratic
BSDE. Our method might be useful to study some (nonlinear) PDEs.
\end{abstract}

\section{Introduction}

In this article, we study positive solutions $u$ of a linear parabolic
equation%
\begin{equation}
\left( L-\frac{\partial }{\partial t}\right) u=0\text{ \ \ in }(0,\infty
)\times M\text{,}  \label{he-hq1}
\end{equation}%
where $M$ is either the Euclidean space $\mathbb{R}^{n}$ and $L$ is an
elliptic or sub-elliptic operator of second-order $L=\frac{1}{2}\sum_{\alpha
=1}^{m}A_{\alpha }^{2}+A_{0}$, $\{A_{0},\cdots ,A_{m}\}$ is a family of
vector fields on $\mathbb{R}^{n}$, or $M$ is a complete manifold of
dimension $n$ with Riemannian metric $(g_{ij})$, and $2L$ is the
Laplace-Beltrami operator%
\begin{equation*}
\Delta =\frac{1}{\sqrt{g}}\sum\limits_{i,j=1}^{n}\frac{\partial }{\partial
x^{i}}\sqrt{g}g^{ij}\frac{\partial }{\partial x^{j}},
\end{equation*}%
where $g$ denotes the determinate of $(g_{ij})$ and $(g^{ij})$ is the
inverse of the matrix $(g_{ij})$.

The well-posedness and the regularity theory for (\ref{he-hq1}) are parts of
the classical theory in partial differential equations, see \cite{MR0241821}%
, \cite{MR0222474} and \cite{MR1465184} for details. On the other hand, it
remains an interesting question to devise precise estimates of a solution $u$
in terms of the (geometric) structures of (\ref{he-hq1}). There is already a
large number of papers devoted to this question. Among many interesting
results, let us cite two of them which are most relevant to the present
paper. The first result is a classical result under the name of semigroup
domination, first discovered by Donnelly and Li \cite{MR654476}, which says
that if the Ricci curvature is bounded from below by $C$, then 
\begin{equation}
|\nabla P_{t}u_{0}|\leq e^{-Ct}P_{t}|\nabla u_{0}|,\text{ \ \ \ \ for \ }%
u_{0}\in C_{b}^{1}(M),  \label{s-d-01}
\end{equation}%
where $(P_{t})_{t\geq 0}$ is the heat semigroup on $M$, so that the
left-hand side is the norm of the gradient of $u(t,\cdot )=P_{t}u_{0}$ a
solution to the heat equation 
\begin{equation}
\left( \frac{1}{2}\Delta -\frac{\partial }{\partial t}\right) u=0\text{ \ \
in }(0,\infty )\times M  \label{hethq1}
\end{equation}%
with initial data $u(0,\cdot )=u_{0}$, while the right-hand side $%
P_{t}|\nabla u_{0}|$ is a solution of (\ref{hethq1}) with initial data $%
|\nabla u_{0}|$. The second result is Li-Yau's estimate first established in 
\cite{MR834612}. If the Ricci curvature is non-negative, and if $u$ is a
positive solution of (\ref{hethq1}) then 
\begin{equation}
|\nabla \log u|^{2}-2\frac{\partial }{\partial t}\log u\leq \frac{n}{t}\text{
\ \ \ for }t>0\text{ .}  \label{gr-01}
\end{equation}%
In fact, in the same paper \cite{MR834612}, Li and Yau also obtained a
gradient estimate for positive solutions in terms of the dimension and a
lower bound (which may be negative) of the Ricci curvature, though less
precise. Their estimates in negative case have been improved over the years,
see for example \cite{MR1305712}, \cite{MR1355703}, \cite{MR1681640} and 
\cite{MR2294794}.

In this paper we prove several  gradient estimates for the positive
solutions of (\ref{he-hq1}). Let us first mention some simple ones for illustration.

\begin{theorem}
\label{in-th1}Let $M$ be a complete manifold with non-negative Ricci
curvature. Suppose $u$ is a positive solution of (\ref{hethq1}) with initial
data $u_{0}>0$, then%
\begin{equation}
|\nabla \log u(t,\cdot )|^{2}\leq \frac{4}{t}||\log u_{0}||_{\infty }, \text{
\ \ \ for }t>0, \label{grad1}
\end{equation}%
where $||\cdot ||_{\infty }$ denotes the $L^{\infty }$ norm on $M$.
\end{theorem}

\begin{remark}
\label{remark:referee}
As pointed out by the referee, (\ref{grad1}) can be derived from the reverse logarithmic Sobolev inequality due to Bakry and Ledoux \cite{MR2294794}.
In fact, from the reverse logarithmic Sobolev inequality,
$$tP_t(u_0)(x)|\log P_t(u_0)|^2(x)\le 2[P_t(u_0\log u_0)(x)-P_t(u_0)(x)\log P_t(u_0)(x)], $$
from which we derive (\ref{grad1}). However, our method is  useful to study estimates for other (nonlinear) PDEs with subelliptic operators, see Theorems \ref{ho-th1} and \ref{ho-th2} in Section 3.
\end{remark}

\begin{remark}
Theorem \ref{in-th1} is very closed to Harnack estimate for the heat equation which is dimension free, see R. Hamilton \cite{Hamilton}. The relation between the Bakry-Ledoux reverse logarithmic Sobolev inequality
and a slight improvement of Hamilton's Harnack inequality was discussed in a very interesting paper by
X. D. Li \cite{LiXiangDong}.
\end{remark}

Indeed we will establish a similar estimate for the heat equation with a
sub-elliptic operator, under similar curvature conditions, and indeed we
will establish a gradient estimate for a complete manifold whose Ricci
curvature is bounded from below. 

\begin{theorem}
\label{in-th2}Let $M$ be a complete manifold of dimension $n$ with
non-negative Ricci curvature. Suppose $u$ is non-negative solution to the
heat equation of (\ref{hethq1}) with initial data $u_{0}>0$.
If $C\in \lbrack 0,\infty ]$ such that $-\Delta \log u_{0}\leq C$, then%
\begin{equation*}
|\nabla \log u|^{2}-2\frac{\partial }{\partial t}\log u\leq \frac{C}{\frac{t%
}{n}C+1}\text{ for }t\geq 0\text{. }
\end{equation*}
\end{theorem}

By setting $C=\infty $ we recover Li-Yau's estimate (\ref{gr-01}).

The novelty of the present paper is not so much about the  gradient
estimates in Theorem \ref{in-th1} and Theorem \ref{in-th2}, what is interesting of the present
work is the approach we are going to develop in order to discover and prove
these gradient estimates. Our approach brings together with the martingale
analysis to the study of a class of non-linear PDEs with quadratic growth.
Of course the connection between the harmonic analysis, potential theory and
martingales is not new, which indeed has a long tradition, standard books
may be mentioned in this aspect, such as \cite{MR0058896}, \cite{MR1814344}, 
\cite{MR750829} and etc., what is new in our study is an interesting
connection between the BMO martingales and positive solutions of the heat
equation (\ref{hethq1}).

To take into account of the positivity, it is better to consider the Hopf
transformation of a positive solution $u$ to (\ref{hethq1}), i.e. $f=\log u$%
, then $f$ itself solves a parabolic equation with quadratic non-linear
term, namely%
\begin{equation}
\left( \frac{1}{2}\Delta -\frac{\partial }{\partial t}\right) f=-\frac{1}{2}%
|\nabla f|^{2}\text{ \ \ \ in }[0,\infty )\times M\text{.}  \label{f-hq1}
\end{equation}

The preceding equation (\ref{f-hq1}) is an archetypical example of a kind of
semi-linear parabolic equations with quadratic growth which has attracted
much attention recently associated with backward stochastic differential
equations, for example Kobylanski \cite{MR1782267}, Briand-Hu \cite%
{MR2257138}, Delbaen et al. \cite{DelbaenHuBao} and etc.

The main idea may be described as the following. Suppose $f$ is a smooth
solution of the non-linear equation (\ref{f-hq1}), and $X_{t}=B_{t}+x$ where 
$B$ is a standard Brownian motion on a complete probability space. Let $%
Y_{t}=f(T-t,X_{t})$ and $Z_{t}=(Z_{t}^{i})$ where $Z_{t}^{i}=\nabla
^{i}f(T-t,X_{t})$, $\nabla ^{i}$ is the covariant derivative written in a
local orthonormal coordinate system. Then, It\^{o}'s lemma applying to $f$
and $X$ may be written as 
\begin{equation}
Y_{T}-Y_{t}=\sum_{i=1}^{n}\int_{t}^{T}Z_{s}^{i}dB_{s}^{i}-\frac{1}{2}%
\int_{t}^{T}|Z_{s}|^{2}ds\text{.}  \label{bs-d1}
\end{equation}%
On the other hand, it was a remarkable discovery by Bismut \cite{MR0453161}
(for a special linear case) and Pardoux-Peng \cite{MR1037747} that given the
terminal random variable $Y_{T}\in L^{2}(\Omega ,\mathcal{F}_{T},\mathbb{P})$%
, there is actually a unique pair $(Y,Z)$ where $Y$ is a continuous
semimartingale and $Z$ is a predictable process which satisfies (\ref{bs-d1}%
). The actual knowledge that $Z$ is the gradient of $Y$ may be restored if $%
Y_{T}=f_{0}(X_{T})$. The backward stochastic differential equation (\ref%
{bs-d1}) with a bounded random terminal $Y_{T}$, which has a non-linear term
of quadratic growth and thus is not covered by Pardoux-Peng \cite{MR1037747}%
, was resolved by Kobylanski \cite{MR1782267}. Observe that the martingale
part of $Y$ is the It\^{o} integral of $Z$ against Brownian motion $B$
(which is denoted by $Z.B$). It can be shown that, if $Y$ is bounded, then $%
Z.B$ is a BMO martingale up to time $T$, so that the exponential martingale 
\begin{equation*}
\mathcal{E}(h(Z).B)_{t}=\exp \left[ \sum_{i=1}^{n}%
\int_{0}^{t}h^{i}(Z_{s})dB_{s}^{i}-\frac{1}{2}\int_{0}^{t}|h(Z_{s})|^{2}ds%
\right]
\end{equation*}%
is a uniformly integrable martingale (up to time $T$), as long as $h$ is
global Lipschitz continuous. The main technical step in our approach is
that, due to the special feature of our non-linear term in (\ref{f-hq1}), we
can choose $h^{i}(z)=z^{i}$ (one has to go through the detailed computations
below to see why this choice of $h^{i}$ is a good one), and making change of
probability measure to $\mathbb{Q}$ by $\frac{d\mathbb{Q}}{d\mathbb{P}}=%
\mathcal{E}(h(Z).B)_{T}$, then, under $\mathbb{Q}$, not only $Z.\tilde{B}$
is again a BMO martingale (where $\tilde{B}$ is the martingale part of $B$
under the new probability $\mathbb{Q}$), but also $t\rightarrow |Z_{t}|^{2}$
is a non-negative submartingale. Next by utilizing the BSDE (\ref{bs-d1}),
we can see the BMO norm of $Z.\tilde{B}$ under $\mathbb{Q}$ is dominated at
most $2\sqrt{||Y||_{\infty }}$, that is%
\begin{equation*}
\mathbb{E}^{\mathbb{Q}}\left\{ \left. \int_{t}^{T}|Z_{s}|^{2}ds\right\vert 
\mathcal{F}_{t}\right\} \leq 4||Y||_{\infty }\text{.}
\end{equation*}%
Finally the sub-martingale property of $|Z_{t}|^{2}$ allows to move $%
|Z_{s}|^{2}$ (for $s\in (t,T)$) out from the time integral on the left-hand
side of the previous inequality, which in turn yields the gradient estimate.

Let us now give a heuristic probabilistic proof to Theorem \ref{in-th2} to
explain from where such estimates come from. Let $f=\log u$, and $G=-\Delta
f $. Then one can show that 
\begin{equation*}
G=|\nabla f|^{2}-2f_{t}
\end{equation*}%
(where $f_{t}$ stands for the time derivative $\frac{\partial }{\partial t}f$
for simplicity). Moreover, $G$ satisfies 
\begin{equation*}
(L-\frac{\partial }{\partial t})G=\frac{1}{n}G^{2}+H\text{,}
\end{equation*}%
where 
\begin{equation*}
H=(|\nabla \nabla f|^{2}-\frac{1}{n}G^{2})+2\text{Ric}(\nabla f,\nabla f)%
\text{,}
\end{equation*}%
and $H\geq 0$. We suppose here $G>0$. Consider the BSDE: 
\begin{equation*}
dY_{t}=Z_{t}dB_{t}+\frac{1}{n}Y_{t}^{2}dt,\quad Y_{T}=G(0,x+B_{T})\text{.}
\end{equation*}%
Then 
\begin{equation*}
Y_{t}\geq G(T-t,x+B_{t})\text{.}
\end{equation*}%
Setting 
\begin{equation*}
U_{t}=\frac{1}{Y_{t}},\quad V_{t}=-\frac{Z_{t}}{Y_{t}^{2}},
\end{equation*}%
then $(U,V)$ satisfies the following quadratic BSDE: 
\begin{equation*}
dU_{t}=-\frac{1}{n}dt+V_{t}dB_{t}+\frac{|V_{t}|^{2}}{U_{t}}dt.
\end{equation*}%
Using BMO martingale techniques, one can prove that there exists a new
probability measure $\mathbb{Q}$ under which $\tilde{B}_{t}=B_{t}+%
\int_{0}^{t}\frac{V_{s}}{U_{s}}ds$ is a Brownian motion. Hence 
\begin{equation*}
dU_{t}=-\frac{1}{n}dt+V_{t}d\tilde{B}_{t},
\end{equation*}%
from which we deduce that $U_{0}=\frac{T}{n}+\mathbb{E}^{\mathbb{Q}}[U_{T}]$%
, and 
\begin{equation*}
Y_{0}=\frac{1}{\frac{T}{n}+\mathbb{E}^{\mathbb{Q}}\left[ \frac{1}{Y_{T}}%
\right] }
\end{equation*}%
which yields the estimate in Theorem \ref{in-th2}.

Even though the above heuristic proof is probabilistic (which can be made
rigorous), we prefer to give a pure analytic proof in the last section.

The paper is organized as follows. Next section is devoted to some basic
facts about quadratic BSDEs including BMO martingales. Section 3 establishes
the gradient estimates for some linear parabolic PDEs on Euclidean space,
while Section 4 establishes these estimates on complete manifold. Last
section is devoted to establish a generalized Li-Yau estimate via analytic
tool.

\section{BSDE\ and BMO martingales}

Let us begin with an interesting result about BSDEs with quadratic growth.
The kind of BSDEs we will deal with in this paper has the following form%
\begin{equation}
dY=\sum_{j=1}^{m}Z^{j}F^{j}(Y,Z)dt+\sum_{j=1}^{m}Z^{j}dB^{j}, \quad Y_T=\xi,
\label{bs-51}
\end{equation}%
with terminal value $\xi\in L^{\infty }(\Omega ,\mathcal{F}_{T},\mathbb{P}) $
which is given, where $B=(B^{1},\cdots ,B^{m})$ is a standard Brownian
motion, $(\mathcal{F}_{t})_{t\geq 0}$ is the Brownian filtration associated
with $B$, and $F^{j}$ are continuous function on $\mathbb{R}\times \mathbb{R}%
^{m}$ with at most linear growth: there is a constant $C_{1}\geq 0$ such that%
\begin{equation*}
|F(y,z)|\leq C_{1}(1+|y|+|z|)\text{ \ \ }\forall (y,z)\in \mathbb{R}\times 
\mathbb{R}^{m}\text{.}
\end{equation*}

According to Peng \cite{MR1149116} and as we have seen in the Introduction,
if $u$ is a bounded smooth solution to the following non-linear parabolic
equation%
\begin{equation}
\frac{\partial }{\partial t}u+\sum_{j=1}^{d}F^{j}(u,\nabla u)\frac{\partial u%
}{\partial x_{j}}=\frac{1}{2}\Delta u\text{ \ \ in }[0,\infty )\times 
\mathbb{R}^{m},  \label{q-pde1}
\end{equation}%
with initial data $u_{0}$, then $Y_{t}=u(T-t,B_{t}+\cdot )$ and $%
Z_{t}=\nabla u(T-t,B_{t}+\cdot )$ is a solution pair of (\ref{bs-51}) with
terminal value $Y_{T}=u_{0}(B_{T}+\cdot )$. The special feature of (\ref%
{q-pde1}) is that the maximum principle applies, which implies that global
solutions (here global means for large $t$) exist for the initial value
problem of the system as long as the initial data is bounded (though, this
constraint can be relaxed a bit, but for the simplicity we content ourself
to the bounded initial data problem). The maximum principle implies that as
long as $u$ is a solution to (\ref{q-pde1}) then $|u(x,t)|\leq
||u_{0}||_{\infty }$. Therefore, if the initial data $u_{0}$ is bounded, and 
$F^{j}$ are global Lipschitz, then, according to Theorem 6.1 on page 592, 
\cite{MR0241821}, $u$ exists for all time, and both $u$ and $\nabla u$ are
bounded on $\mathbb{R}^{m}\times \lbrack 0,T]$.

The maximum principle for (\ref{bs-51}) however remains true even for a
bounded random terminal value (so called non Markovian case), which in turn
yields that the martingale part of $Y$ is a BMO martingale. This is the
context of the following

\begin{proposition}
\label{BMOpr1} Suppose that $\xi \in L^{\infty }(\Omega ,\mathcal{F}_{T},%
\mathbb{P})$. There exists a unique solution $(Y,Z)$ to (\ref{bs-51}) such
that $Y$ is bounded and $M=Z.B$ is a square integrable martingale. Moreover $%
M=Z.B$ is a BMO martingale up to time $T$, and 
\begin{equation*}
||Y(t)||_{\infty }\leq ||\xi ||_{\infty }\text{ \ \ }\forall t\in \lbrack
0,T]\text{.}
\end{equation*}
\end{proposition}

\begin{proof}
The existence and uniqueness is already given in \cite{MR1782267}. The fact
that $M=Z.B$ is a BMO martingale up to time $T$ is proved in \cite{Morlais}.
Then there exists a constant $C_{2}>0$ such that 
\begin{equation*}
\mathbb{E}\left[ \int_{t}^{T}|Z_{s}|^{2}ds\Big |{\cal F}_t\right] \leq C_{2}\text{.}
\end{equation*}

Let $N_{t}=\sum_{j=1}^{d}\int_{0}^{t}F^{j}(Y_{s},Z_{s})dB_{s}^{j}$. Since%
\begin{eqnarray*}
\langle N,N\rangle _{T}-\langle N,N\rangle _{t}
&=&\int_{t}^{T}\sum_{j}|F^{j}(Y_{s},Z_{s})|^{2}ds \\
&\leq &\int_{t}^{T}C_{1}^{2}(1+|Y_{s}|+|Z_{s}|)^{2}ds,
\end{eqnarray*}%
so there exists a constant $C_{3}>0$ such that 
\begin{equation*}
\mathbb{E}\left\{ \left. \langle N,N\rangle _{T}-\langle N,N\rangle
_{t}\right\vert \mathcal{F}_{t}\right\} \leq C_{3}\text{.}
\end{equation*}%
Therefore $N$ is a BMO martingale. Hence the stochastic exponential $%
\mathcal{E}(-N)$ is a martingale up to $T$.

Define a probability measure $\mathbb{Q}$ on $(\Omega ,\mathcal{F}_{T})$ by $%
d\mathbb{Q}/d\mathbb{P}=\mathcal{E}(-N)_{T}$. Then, according to Girsanov's
theorem $\tilde{B}_{t}=B_{t}+\langle N,B\rangle _{t}$ is a Brownian motion
up to time $T$ under $\mathbb{Q}$, and $(Y,Z)$ is a solution to the simple
BSDE 
\begin{equation*}
dY_{t}=Z_{t}.d\tilde{B}_{t}
\end{equation*}%
under the probability $\mathbb{Q}$, whose solution is given by%
\begin{equation}
Y_{t}=\mathbb{E}^{\mathbb{Q}}\{\xi |\mathcal{F}_{t}\}=\mathbb{E}\{\mathcal{E}%
(-N)_{T}\mathcal{E}(-N)_{t}^{-1}\xi |\mathcal{F}_{t}\}\text{ \ \ for }t\leq T%
\text{.}  \label{re-01}
\end{equation}%
It particularly  implies that $||Y_{t}||_{\infty }\leq ||\xi ||_{\infty }$.
\end{proof}

\section{Stochastic flows and gradient estimates}

Let $A_{0}$, $A_{1}$, $\cdots $, $A_{m}$ be $m+1$ smooth vector fields on
Euclidean space $\mathbb{R}^{n}$, where $n$ is a non-negative integer. Then,
we may form a sub-elliptic differential operator of second order in $\mathbb{%
R}^{n}$: 
\begin{equation}
L=\frac{1}{2}\sum_{\alpha =1}^{m}A_{\alpha }^{2}+A_{0}, \label{ge-hq1}
\end{equation}%
here we add a factor $\frac{1}{2}$ in order to save the constant $\sqrt{2}$
in front of Brownian motion which will appear frequently in computations in
the remaining of the paper. Our goal is to devise an explicit gradient
estimate for a (smooth) positive solution $u$ of the heat equation%
\begin{equation}
\left( L-\frac{\partial }{\partial t}\right) u=0\text{, \ on }(0,\infty
)\times \mathbb{R}^{n}\text{,}  \label{l-heat-hq1}
\end{equation}%
by utilizing the BSDE associated with the Hopf transformation $f=\log u$,
which satisfies the semi-linear parabolic equation%
\begin{equation}
\left( L-\frac{\partial }{\partial t}\right) f=-\frac{1}{2}\sum_{\alpha
=1}^{m}|A_{\alpha }f|^{2}\text{, \ on }(0,\infty )\times \mathbb{R}^{n}\text{
.}  \label{eqf-hq1}
\end{equation}

\subsection{Stochastic flow}

The first ingredient in our approach is the theory of stochastic flows
defined by the following stochastic differential equation 
\begin{equation}
d\varphi =A_{0}(\varphi )dt+\sum_{\alpha =1}^{m}A_{\alpha }(\varphi )\circ
dw^{\alpha }\text{, }\varphi (0,\cdot )=x,  \label{sde-hq1}
\end{equation}%
where $\circ d$ denotes the Stratonovich differential, developed by
Baxendale \cite{MR762306}, Bismut \cite{MR629977}, Eells and Elworthy \cite%
{MR0516923}, Malliavin \cite{MR540035}, Kunita \cite{MR580134} and etc. The
reader may refer to Ikeda and Watanabe \cite{MR637061} for a definite
account. To ensure the global existence of a stochastic flow, we require the
following condition to be satisfied.

\begin{condition}
\label{co-hq1}Let $A_{\alpha }=\sum_{j=1}^{n}A_{\alpha }^{j}\frac{\partial }{%
\partial x^{j}}$. Assume that $A_{\alpha }^{j}$ have bounded derivatives.
\end{condition}

By writing (\ref{sde-hq1}) in terms of It\^{o}'s stochastic integrals, namely%
\begin{equation}
d\varphi ^{j}=\left[ A_{0}^{j}+\frac{1}{2}\sum_{\alpha =1}^{m}A_{\alpha }^{i}%
\frac{\partial A_{\alpha }^{j}}{\partial x^{i}}\right] (\varphi
)dt+\sum_{\alpha =1}^{m}A_{\alpha }^{j}(\varphi )dw^{\alpha }\text{, }%
\varphi (0,\cdot )=x\text{,}  \label{sde-hq2}
\end{equation}%
where (and thereafter) Einstein's summation convention has been used:
repeated indices such as $l$ is summed up from $1$ up to $n$. The existence
and uniqueness of a strong solution follow directly from the standard result
in It\^{o}'s\ theory, which in turn determines a diffusion process in $%
\mathbb{R}^{n}$ with the infinitesimal generator \thinspace $L$.

In fact, more can be said about the unique strong solution, and important
consequences are collected here which will be used later on. Suppose $%
w=(w_{t})$ is a standard Brownian motion (started at $0$) with its Brownian
filtration $(\mathcal{F}_{t})_{t\geq 0}$ on the classical Wiener space $%
(\Omega ,\mathcal{F},\mathbb{P})$ of dimension $m$, so that $%
w=(w_{t})_{t\geq 0}$ is the coordinate process on the space $\Omega $ of
continuous paths in $\mathbb{R}^{m}$ with initial zero. Then, there is a
measurable mapping $\varphi :\mathbb{R}^{+}\times \Omega \times \mathbb{R}%
^{n}\longrightarrow \mathbb{R}^{n}$ and a probability null set $\mathcal{N}$%
, which possess the following properties.

\begin{enumerate}
\item $w\rightarrow \varphi (t,w,x)$ is $\mathcal{F}_{t}$-measurable for $%
t\geq 0$ and $x\in \mathbb{R}^{n}$, and $\varphi (0,w,x)=x$ for every $w\in
\Omega \setminus \mathcal{N}$ and $x\in \mathbb{R}^{n}$.

\item $t\rightarrow \varphi (t,w,x)$ is continuous, that is $\varphi (\cdot
,w,x)\in C(\mathbb{R}^{+},\mathbb{R}^{n})$, for $w\in \Omega \setminus 
\mathcal{N}$ and $x\in \mathbb{R}^{n}$. $t\rightarrow \varphi (t,\cdot ,x)$
is a continuous semimartingale for any $x\in \mathbb{R}^{n}$.

\item $x\rightarrow \varphi (t,w,x)$ is a diffeomorphism of $\mathbb{R}^{n}$
for each $w\in \Omega \setminus \mathcal{N}$ and $t\geq 0$. That is $%
x\rightarrow \varphi (t,w,x)$ is smooth and its inverse exists, and the
inverse is also smooth.

\item The family $\{\varphi (t,\cdot ,x):t\geq 0,x\in \mathbb{R}^{n}\}$ is a
stochastic flow:%
\begin{equation*}
\varphi (t+s,w,x)=\varphi (t,\theta _{s}w,\varphi (s,w,x))
\end{equation*}%
for all $t,s\geq 0$, $x\in \mathbb{R}^{n}$ and $w\in \Omega \setminus 
\mathcal{N}$, where $\theta _{s}:\Omega \rightarrow \Omega $ is the shift
operator sending a path $w$ to a path $\theta _{s}w(t)=w(t+s)$ for $t\geq 0$.

\item For each $x\in \mathbb{R}^{n}$, $\varphi (t)=\varphi (t,\cdot ,x)$ (or
denoted by $\varphi (t,x)$) is the unique strong solution of (\ref{sde-hq1}).

\item Let $J_{j}^{i}(t,w,x)=\frac{\partial \varphi ^{i}(t,w,x)}{\partial
x^{j}}$ for $i,j\leq n$. Then $J_{j}^{i}(0,w,x)=\delta _{j}^{i}$ and $J$
solves the following SDE%
\begin{equation}
dJ_{j}^{i}=\frac{\partial A_{0}^{i}}{\partial x^{l}}(\varphi
)J_{j}^{l}dt+\sum_{\alpha =1}^{m}\frac{\partial A_{\alpha }^{i}}{\partial
x^{l}}(\varphi )J_{j}^{l}\circ dw^{\alpha }\text{, \ }J_{j}^{i}(0)=\delta
_{j}^{i},  \label{der-sde1}
\end{equation}%
and its inverse matrix $K=J^{-1}=(K_{j}^{i})$ solves 
\begin{equation}
dK_{j}^{i}=-K_{l}^{i}\frac{\partial A_{0}^{l}}{\partial x^{j}}(\varphi
)dt-\sum_{\alpha =1}^{m}K_{l}^{i}\frac{\partial A_{\alpha }^{l}}{\partial
x^{j}}(\varphi )\circ dw^{\alpha }\text{, \ }K_{j}^{i}(0)=\delta _{j}^{i}%
\text{.}  \label{der-sde2}
\end{equation}
\end{enumerate}

In our computations below, we have to use It\^{o}'s integrals rather than
Stratonovich's ones. Therefore we would like to rewrite (\ref{der-sde1}, \ref%
{der-sde2}) in terms of It\^{o}'s differential, so%
\begin{eqnarray}
dJ_{j}^{i} &=&\sum_{\alpha =1}^{m}\frac{\partial A_{\alpha }^{i}}{\partial
x^{l}}(\varphi )J_{j}^{l}dw^{\alpha }  \notag \\
&&+\left[ \frac{\partial A_{0}^{i}}{\partial x^{l}}+\frac{1}{2}\sum_{\alpha
=1}^{m}\left( A_{\alpha }^{k}\frac{\partial ^{2}A_{\alpha }^{i}}{\partial
x^{l}\partial x^{k}}+\frac{\partial A_{\alpha }^{k}}{\partial x^{l}}\frac{%
\partial A_{\alpha }^{i}}{\partial x^{k}}\right) \right] (\varphi
)J_{j}^{l}dt,  \label{e-ito-hq1}
\end{eqnarray}%
and%
\begin{eqnarray}
dK_{j}^{i} &=&-\sum_{\alpha =1}^{m}K_{l}^{i}\frac{\partial A_{\alpha }^{l}}{%
\partial x^{j}}(\varphi )dw^{\alpha }  \notag \\
&&-K_{l}^{i}\left[ \frac{\partial A_{0}^{l}}{\partial x^{j}}+\frac{1}{2}%
\sum_{\alpha =1}^{m}\left( A_{\alpha }^{k}\frac{\partial ^{2}A_{\alpha }^{l}%
}{\partial x^{j}\partial x^{k}}-\frac{\partial A_{\alpha }^{k}}{\partial
x^{j}}\frac{\partial A_{\alpha }^{l}}{\partial x^{k}}\right) \right]
(\varphi )dt\text{.}  \label{f-ito-hq1}
\end{eqnarray}

\subsection{Structure assumptions}

We introduce some technical assumptions on the structure of the Lie algebra
generated by the family of vector fields $\{A_{0},A_{1},\cdots ,A_{m}\}$, in
addition to Condition \ref{co-hq1}. Recall that $A_{\alpha }=A_{\alpha }^{j}%
\frac{\partial }{\partial x^{j}}$, and $A_{\alpha ,\beta }^{j}$, $A_{\alpha
,\beta ,\gamma }^{j}$ etc. are the corresponding coefficients in Lie
brackets 
\begin{equation*}
\left[ A_{\alpha },A_{\beta }\right] =A_{\alpha ,\beta }^{j}\frac{\partial }{%
\partial x^{j}}\text{, \ }\left[ A_{\alpha },[A_{\beta },A_{\gamma }]\right]
=A_{\alpha ,\beta ,\gamma }^{j}\frac{\partial }{\partial x^{j}}\text{ etc.,}
\end{equation*}%
where
\begin{equation*}
A_{\alpha ,\beta }^{j}=A_{\alpha }^{i}\frac{\partial A_{\beta }^{j}}{%
\partial x^{i}}-A_{\beta }^{i}\frac{\partial A_{\alpha }^{j}}{\partial x^{i}}%
\text{ ,}
\end{equation*}%
\begin{eqnarray}
A_{\beta ,\beta ,\alpha }^{k} &=&A_{\beta }^{j}A_{\beta }^{i}\frac{\partial
^{2}A_{\alpha }^{k}}{\partial x^{i}\partial x^{j}}-A_{\beta }^{j}A_{\alpha
}^{i}\frac{\partial ^{2}A_{\beta }^{k}}{\partial x^{i}\partial x^{j}}%
+A_{\beta }^{i}\frac{\partial A_{\alpha }^{k}}{\partial x^{j}}\frac{\partial
A_{\beta }^{j}}{\partial x^{i}}  \notag \\
&&-2A_{\beta }^{j}\frac{\partial A_{\alpha }^{i}}{\partial x^{j}}\frac{%
\partial A_{\beta }^{k}}{\partial x^{i}}+A_{\alpha }^{i}\frac{\partial
A_{\beta }^{j}}{\partial x^{i}}\frac{\partial A_{\beta }^{k}}{\partial x^{j}}%
\text{ .}  \label{3-c-hq1}
\end{eqnarray}%
etc. Let%
\begin{equation}
R_{\alpha }^{k}=\sum_{\beta =1}^{m}A_{\beta ,\beta ,\alpha }^{k}=\sum_{\beta
=1}^{m}[A_{\beta },[A_{\beta },A_{\alpha }]]^{k}\text{ .}  \label{c-hq1-1}
\end{equation}

\begin{condition}
\label{co-hq2} There is a constant $C_{1}\geq 0$ such that for any $\xi
=(\xi _{i})_{i\leq n}$, $\theta _{\beta }=(\theta _{i,\beta })_{i\leq n}\in 
\mathbb{R}^{n}$ ($\beta =1,\cdots ,m$), it holds that%
\begin{eqnarray*}
&&\sum_{\alpha ,\beta =1}^{m}\left( \sum_{k=1}^{n}A_{\alpha }^{k}\theta
_{k,\beta }\right) ^{2}+2\sum_{\alpha ,\beta =1}^{m}\left( \sum_{k=1}^{n}\xi
_{k}A_{\beta ,\alpha }^{k}\right) \left( \sum_{i=1}^{n}A_{\alpha }^{i}\theta
_{i,\beta }\right) \\
&&+2\sum_{\alpha ,\beta =1}^{m}\left( \sum_{k=1}^{n}\xi _{k}A_{\alpha
}^{k}\right) \left( \sum_{i=1}^{n}A_{\beta ,\alpha }^{i}\theta _{i,\beta
}\right) \\
&\geq &-C_{1}\sum_{\alpha =1}^{m}\left( \sum_{k=1}^{n}A_{\alpha }^{k}\xi
_{k}\right) ^{2}
\end{eqnarray*}%
i.e.%
\begin{eqnarray*}
&&\sum_{\alpha ,\beta =1}^{m}\left( \langle A_{\alpha },\theta _{\beta
}\rangle ^{2}+2\langle A_{\alpha },\xi \rangle \langle A_{\beta ,\alpha
},\theta _{\beta }\rangle +2\langle A_{\alpha },\theta _{\beta }\rangle
\langle A_{\beta ,\alpha },\xi \rangle \right) \\
&\geq &-C_{1}\sum_{\alpha =1}^{m}\langle A_{\alpha },\xi \rangle ^{2}\text{.}
\end{eqnarray*}
\end{condition}

\begin{condition}
\label{co-hq3}There is a constant $C_{2}\ge 0$ such that for any $\xi =(\xi
_{i})_{i\leq n}\in \mathbb{R}^{n}$ 
\begin{eqnarray*}
&&\sum_{i,k=1}^{n}\xi _{i}\left( \sum_{\alpha =1}^{m}(A_{\alpha
}^{i}R_{\alpha }^{k}+2A_{\alpha }^{i}A_{0,\alpha }^{k})+\sum_{\alpha ,\beta
=1}^{m}A_{\beta ,\alpha }^{i}A_{\beta ,\alpha }^{k}\right) \xi _{k} \\
&\geq &-C_{2}\sum_{\alpha =1}^{m}\left( \sum_{k=1}^{n}A_{\alpha }^{k}\xi
_{k}\right) ^{2}\text{.\ }
\end{eqnarray*}
\end{condition}

\begin{remark} Condition (\ref{co-hq1}) is standard in the literature, while 
Conditions (\ref{co-hq2}) and (\ref{co-hq3}) are satisfied if $A$ is elliptic 
or $A$ satisfies the Frobenius integrability condition.
\end{remark}

Let us  suppose the following Frobenius integrability condition: there exist some bounded
smooth coefficients $c_{\beta,\alpha}^l(x)$, such that 
\begin{equation*}
A_{\beta,\alpha}=\sum_{l=1}^m c_{\beta,\alpha}^lA_l,\quad
\beta=0,1,\dots,m,\quad \alpha=1,\dots,m.
\end{equation*}
In other words, the Lie brackets $A_{\beta,\alpha}, \quad
\beta=0,1,\dots,m,\quad \alpha=1,\dots,m,$ must lie in the linear span of $%
A_1,\dots,A_m$. Then the conditions (3.2) and (3.3) are satisfied.

Indeed, 
\begin{eqnarray*}
[A_{\beta}, [A_\beta,A_\alpha]]^j &=& A_\beta^i\frac{\partial}{\partial x^i}%
(c_{\beta,\alpha}^l A_l^j)-c_{\beta,\alpha}^l A_l^i\frac{\partial A_\beta^j}{%
\partial x^i} \\
&=& c_{\beta,\alpha}^l A_{\beta,l}^j+\frac{\partial c_{\beta,\alpha}^l}{%
\partial x^i}A_\beta^i A_l^j \\
&=& (c_{\beta,\alpha}^l c_{\beta,l}^k+\frac{\partial c_{\beta,\alpha}^k}{%
\partial x^i}A_\beta^i)A_k^j.
\end{eqnarray*}
This means that $[A_\beta, [A_\beta,A_\alpha]]$ also lies in the linear span
of $A_1,\dots,A_m$, and the conditions (3.2) and (3.3) are easily checked.


\subsection{The density processes $Z_{\protect\alpha }$}

Let us consider a smooth solution $f$ to the following non-linear parabolic
equation 
\begin{equation}
\left( L-\frac{\partial }{\partial t}\right) f=h(f,A_{\alpha }f)\text{, \ on 
}\mathbb{R}_{+}\times \mathbb{R}^{n}  \label{eqf-hq2},
\end{equation}%
where $h$ is a $C^{1}$-function on $\mathbb{R}\times \mathbb{R}^{m}$, though
our archetypical example is $f=\log u$ and $u$ is a positive solution to
equation (\ref{eqf-hq1}).

By It\^{o}'s formula,
\begin{equation}
Y_{t}=Y_{T}-\sum_{\alpha =1}^{m}\int_{t}^{T}Z^{\alpha }dw^{\alpha
}-\int_{t}^{T}h(Y,Z)ds,  \label{ito-hq1}
\end{equation}%
where%
\begin{eqnarray*}
Y(t,w,x) &=&f(T-t,\varphi (t,w,x))\text{, } \\
Z^{\alpha }(t,w,x) &=&\left( A_{\alpha }f\right) (T-t,\varphi (t,w,x))
\end{eqnarray*}%
for $\alpha =1,\cdots ,m$, and $Z=(Z^{\alpha })$. The arguments $w$ and / or 
$x$ will be suppressed if no confusion may arise. Equivalently 
\begin{equation}
dY=\sum_{\alpha =1}^{m}Z^{\alpha }dw^{\alpha }+h(Y,Z)dt\text{.}
\label{bs-hq91}
\end{equation}

Our aim in this part is to show that $Z$ is an It\^{o} process, and derives
stochastic differential equations for $Z$ (which in turn gives its
Doob-Meyer's decomposition).

It is clear that both $Y$ and $Z^{\alpha }$ ($\alpha =1,\cdots ,m$) are
continuous semimartingales. Taking derivatives with respect to $x^{i}$ ($%
i=1,\cdots ,n$) in the equation (\ref{bs-hq91}) one obtains%
\begin{equation}
dY_{i}=\sum_{\alpha =1}^{m}Z_{i}^{\alpha }dw^{\alpha }+\left(
h_{y}(Y,Z)Y_{i}+\sum_{\alpha =1}^{m}h_{z_{\alpha }}(Y,Z)Z_{i}^{\alpha
}\right) dt,  \label{d-bsde-hq91}
\end{equation}%
where 
\begin{equation*}
Y_{i}=\frac{\partial }{\partial x^{i}}Y\text{ \ \ \ and \ \ }Z_{i}^{\alpha }=%
\frac{\partial }{\partial x^{i}}Z^{\alpha }\text{, \ \ \ \ }i=1,\cdots ,n,%
\text{ }
\end{equation*}%
and 
\begin{equation*}
h_{y}=\frac{\partial }{\partial y}h(y,z)\text{ , \ \ \ }h_{z_{\alpha }}=%
\frac{\partial }{\partial z_{\alpha }}h(y,z)\text{, \ }\alpha =1,\cdots ,m%
\text{.}
\end{equation*}%
On the other hand, by definition,%
\begin{equation*}
Y_{i}(t,\cdot ,x)=\frac{\partial }{\partial x^{i}}Y(t,\cdot ,x)=\frac{%
\partial f}{\partial \varphi ^{j}}(t,\varphi (t,\cdot ,x))J_{i}^{j}(t,\cdot
,x),
\end{equation*}%
so%
\begin{equation*}
\frac{\partial f}{\partial \varphi ^{k}}(t,\varphi (t,\cdot
,x))=K_{k}^{l}(t,\cdot ,x)Y_{l}(t,\cdot ,x)\text{. }
\end{equation*}%
It follows that%
\begin{eqnarray}
Z^{\alpha }(t,\cdot ,x) &=&A_{\alpha }^{k}(\varphi (t,\cdot ,x))\frac{%
\partial f}{\partial \varphi ^{k}}(\varphi (t,\cdot ,x))  \notag \\
&=&A_{\alpha }^{k}(\varphi (t,\cdot ,x))K_{k}^{l}(t,\cdot ,x)Y_{l}(t,\cdot
,x),  \label{z-hq1}
\end{eqnarray}%
which implies that $Z$ is a continuous semimartingale. The equation (\ref%
{z-hq1}) is not new, and has been used by many authors in different contexts.

We next would like to write down the stochastic differential equations that $%
Z^{\alpha }$ must satisfy by using the relation (\ref{z-hq1}), which is
however an easy exercise on integration by parts.\ Indeed, we have

\begin{eqnarray}
dZ^{\alpha } &=&Y_{l}K_{k}^{l}\frac{\partial A_{\alpha }^{k}}{\partial x^{j}}%
(\varphi )d\varphi ^{j}+Y_{l}A_{\alpha }^{k}(\varphi )dK_{k}^{l}+A_{\alpha
}^{k}(\varphi )K_{k}^{l}dY_{l}  \notag \\
&&+Y_{l}\frac{\partial A_{\alpha }^{k}}{\partial x^{j}}(\varphi )d\langle
\varphi ^{j},K_{k}^{l}\rangle +K_{k}^{l}\frac{\partial A_{\alpha }^{k}}{%
\partial x^{j}}(\varphi )d\langle \varphi ^{j},Y_{l}\rangle  \notag \\
&&+A_{\alpha }^{k}(\varphi )d\langle K_{k}^{l},Y_{l}\rangle +\frac{1}{2}%
Y_{l}K_{k}^{l}\frac{\partial ^{2}A_{\alpha }^{k}}{\partial x^{i}\partial
x^{j}}(\varphi )d\langle \varphi ^{i},\varphi ^{j}\rangle \text{ .}
\label{z-sde1}
\end{eqnarray}%
Using the SDEs (\ref{sde-hq1}, \ref{f-ito-hq1}) and the BSDE (\ref%
{d-bsde-hq91}), through a lengthy but completely elementary computation, we
establish the following Doob-Meyer's decomposition for $Z$%
\begin{eqnarray}
dZ^{\alpha } &=&\sum_{\beta =1}^{m}U_{\alpha ,\beta }\left( dw^{\beta
}+h_{z_{\beta }}(Y,Z)dt\right) +K_{k}^{l}\sum_{\beta =1}^{m}A_{\beta ,\alpha
}^{k}Z_{l}^{\beta }dt  \notag \\
&&+K_{k}^{l}Y_{l}\left[ A_{0,\alpha }^{k}+\frac{1}{2}\sum_{\beta
=1}^{m}A_{\beta ,\beta ,\alpha }^{k}+A_{\alpha }^{k}h_{y}(Y,Z)-\sum_{\beta
=1}^{m}A_{\beta ,\alpha }^{k}h_{z_{\beta }}(Y,Z)\right] dt,  \label{z-sde2}
\end{eqnarray}%
where repeated indices are added up from $1$ to $n$, 
\begin{equation}
U_{\alpha ,\beta }=K_{k}^{l}\left( A_{\beta ,\alpha }^{k}Y_{l}+A_{\alpha
}^{k}Z_{l}^{\beta }\right)  \label{u-hq1}
\end{equation}%
and 
\begin{equation*}
A_{\beta ,\alpha }^{k}=[A_{\beta },A_{\alpha }]^{k}\text{, }A_{\beta ,\beta
,\alpha }^{k}=[A_{\beta },[A_{\beta },A_{\alpha }]]^{k}\text{ .}
\end{equation*}

We are now in a position to work out the Doob-Meyer's decomposition for 
\begin{equation*}
|Z|^{2}=\sum_{\alpha =1}^{m}|Z^{\alpha }|^{2}
\end{equation*}%
which simply follows from It\^{o}'s formula and (\ref{z-sde2}). In order to
simplify our displayed formula, we introduce the following notations:%
\begin{equation*}
\xi _{i}=\sum_{j=1}^{n}K_{i}^{j}Y_{j}\text{, }\theta _{k,\beta
}=\sum_{j=1}^{n}K_{k}^{j}Z_{j}^{\beta }\text{ ,}
\end{equation*}%
\ for $i,k=1,\cdots ,n$ and $\beta =1,\cdots ,m$, so that 
\begin{equation}
Z^{\alpha }=\sum_{j=1}^{n}A_{\alpha }^{j}\xi _{j}\text{ \ and }%
|Z|^{2}=\sum_{\alpha =1}^{m}\left( \sum_{j=1}^{n}A_{\alpha }^{j}\xi
_{j}\right) ^{2}\text{ .}  \label{zr-hq1}
\end{equation}

Let 
\begin{equation*}
d\tilde{w}^{\beta }=dw^{\beta }+h_{z_{\beta }}(Y,Z)dt,
\end{equation*}%
which is a Brownian motion under probability $\mathbb{Q}$ with the
Cameron-Martin density%
\begin{equation}
\left. \frac{d\mathbb{Q}}{d\mathbb{P}}\right\vert _{\mathcal{F}_{t}}=\exp %
\left[ -\sum_{\beta =1}^{m}\int_{0}^{t}h_{z_{\beta }}(Y,Z)dw^{\beta }-\frac{1%
}{2}\int_{0}^{t}\sum_{\beta =1}^{m}|h_{z_{\beta }}(Y,Z)|^{2}ds\right] \text{.%
}  \label{cm-hq1}
\end{equation}

Then, by an elementary computation, 
\begin{eqnarray}
d|Z|^{2} &=&2\sum_{\alpha ,\beta =1}^{m}\left( Z^{\alpha }\xi _{k}A_{\beta
,\alpha }^{k}+\xi _{i}A_{\alpha }^{i}A_{\alpha }^{k}\theta _{k,\beta
}\right) d\tilde{w}^{\beta }  \notag \\
&&+\left[ \sum_{\alpha ,\beta =1}^{m}A_{\alpha }^{k}\theta _{k,\beta
}A_{\alpha }^{i}\theta _{i,\beta }+2\xi _{k}\sum_{\alpha ,\beta
=1}^{m}\left( A_{\beta ,\alpha }^{k}A_{\alpha }^{i}+A_{\beta ,\alpha
}^{i}A_{\alpha }^{k}\right) \theta _{i,\beta }\right] dt  \notag \\
&&+2\left[ h_{y}(Y,Z)\sum_{\alpha =1}^{m}A_{\alpha }^{i}A_{\alpha
}^{k}-\sum_{\alpha ,\beta =1}^{m}A_{\alpha }^{i}A_{\beta ,\alpha
}^{k}h_{z_{\beta }}(Y,Z)\right] \xi _{i}\xi _{k}dt  \notag \\
&&+\left[ \sum_{\alpha =1}^{m}A_{\alpha }^{i}\left( R_{\alpha
}^{k}+2A_{0,\alpha }^{k}\right) +\sum_{\alpha ,\beta =1}^{m}A_{\beta ,\alpha
}^{i}A_{\beta ,\alpha }^{k}\right] \xi _{i}\xi _{k}dt\text{.}
\label{z-2-hq1}
\end{eqnarray}

\begin{lemma}
If $h(Y,Z)=h(Y,|Z|^{2})$, then%
\begin{eqnarray}
d|Z|^{2} &=&2\sum_{\alpha ,\beta =1}^{m}\left( Z^{\alpha }\xi _{k}A_{\beta
,\alpha }^{k}+\xi _{i}A_{\alpha }^{i}A_{\alpha }^{k}\theta _{k,\beta
}\right) d\tilde{w}^{\beta }  \notag \\
&&+\left[ \sum_{\alpha ,\beta =1}^{m}A_{\alpha }^{k}\theta _{k,\beta
}A_{\alpha }^{i}\theta _{i,\beta }+2\xi _{k}\sum_{\alpha ,\beta
=1}^{m}\left( A_{\beta ,\alpha }^{k}A_{\alpha }^{i}+A_{\beta ,\alpha
}^{i}A_{\alpha }^{k}\right) \theta _{i,\beta }\right] dt  \notag \\
&&+\left[ \sum_{\alpha =1}^{m}A_{\alpha }^{i}\left( R_{\alpha
}^{k}+2A_{0,\alpha }^{k}+2A_{\alpha }^{k}h_{y}(Y,|Z|^{2})\right)
+\sum_{\alpha ,\beta =1}^{m}A_{\beta ,\alpha }^{i}A_{\beta ,\alpha }^{k}%
\right] \xi _{i}\xi _{k}dt\text{.}  \label{z-2-hq2}
\end{eqnarray}
\end{lemma}

\begin{proof}
In this case%
\begin{eqnarray*}
\sum_{i,k,\alpha ,\beta }\xi _{i}A_{\alpha }^{i}A_{\beta ,\alpha
}^{k}h_{z_{\beta }}(Y,Z)\xi _{k} &=&2\sum_{i,k,\alpha ,\beta }h^{\prime }\xi
_{i}A_{\alpha }^{i}A_{\beta ,\alpha }^{k}A_{\beta }^{l}\xi _{l}\xi _{k} \\
&=&-2\sum_{i,k,\alpha ,\beta }h^{\prime }\xi _{l}A_{\beta }^{l}A_{\beta
,\alpha }^{k}A_{\alpha }^{i}\xi _{i}\xi _{k},
\end{eqnarray*}%
so%
\begin{equation*}
\sum_{i,k,\alpha ,\beta }\xi _{i}A_{\alpha }^{i}A_{\beta ,\alpha
}^{k}h_{z_{\beta }}(Y,Z)\xi _{k}=0,
\end{equation*}%
and thus (\ref{z-2-hq2}) follows directly from (\ref{z-2-hq1}).
\end{proof}

\subsection{Gradient estimates}

Recall that $f$ is a smooth solution to the non-linear parabolic equation (%
\ref{eqf-hq2}), where the nonlinear term $h(Y,Z)$ has at most quadratic
growth. In order to devise explicit estimate for $A_{\alpha }f$ ($\alpha
=1,\cdots ,m$), we assume the following condition to be satisfied.

\begin{condition}
\label{co-hq4} $h(y,z)$ depends only on $(y,|z|^{2})$, i.e. there is a
continuously differentiable function denoted again by $h$ so that $%
h(y,z)=h(y,|z|^{2})$, and we assume that 
\begin{equation*}
\frac{\partial }{\partial y}h(y,z)\geq 0\text{.}
\end{equation*}
\end{condition}

Then, under Conditions (\ref{co-hq1}, \ref{co-hq2}, \ref{co-hq3}, \ref{co-hq4}), 
according to (\ref{z-2-hq2}), we have%
\begin{equation}
d|Z|^{2}\geq -K|Z|^{2}dt+2\left( Z^{\alpha }\xi _{k}A_{\beta ,\alpha
}^{k}+\xi _{i}A_{\alpha }^{i}A_{\alpha }^{k}\theta _{k,\beta }\right) d%
\tilde{w}^{\beta }, \label{su-m-hq2}
\end{equation}%
where $K=C_{1}+C_{2}$.

\begin{lemma}
Assume that Conditions (\ref{co-hq1}, \ref{co-hq2}, \ref{co-hq3}, \ref%
{co-hq4}) are satisfied. Then $M_{t}=e^{Kt}|Z_{t}|^{2}$ is submartingale
under the probability $\mathbb{Q}$ (up to terminal time $T$): 
\begin{equation}
\mathbb{E}^{\mathbb{Q}}\left\{ M_{t}|\mathcal{F}_{s}\right\} \geq M_{s},\text{
\ \ \ }\forall 0\leq s<t\leq T\text{.}  \label{sm-hq1}
\end{equation}
\end{lemma}

\begin{proof}
By It\^{o}'s formula%
\begin{equation*}
dM=KMdt+e^{Kt}d|Z|^{2},
\end{equation*}%
hence, for any $0\leq s\leq t\leq T$, we have%
\begin{eqnarray*}
M_{t}-M_{s} &=&K\int_{s}^{t}M_{r}dr+\int_{s}^{t}e^{Kr}d|Z|^{2} \\
&\geq &2\int_{s}^{t}e^{Ks}\sum_{\alpha ,\beta =1}^{m}\left( Z^{\alpha }\xi
_{k}A_{\beta ,\alpha }^{k}+\xi _{i}A_{\alpha }^{i}A_{\alpha }^{k}\theta
_{k,\beta }\right) d\tilde{w}^{\beta },
\end{eqnarray*}%
which yields (\ref{sm-hq1}).
\end{proof}

We are now in a position to prove the following gradient estimate.

\begin{theorem}
\label{ho-th1}Assume that Conditions (\ref{co-hq1}, \ref{co-hq2}, \ref%
{co-hq3}) are satisfied. Then 
\begin{equation}
\sum_{\alpha =1}^{m}|A_{\alpha }\log u(t,x)|^{2}\leq \frac{4K}{1-e^{-K(T-t)}}%
||\log u_{0}||_{\infty }, \label{est-o1}
\end{equation}%
for any positive solution $u$ of (\ref{l-heat-hq1}).
\end{theorem}

\begin{proof}
Apply the computations in the preceding sub-section to $f=\log u$, and $%
h(y,z)=-\frac{1}{2}|z|^{2}$. Then, under the probability $\mathbb{Q}$
(defined by (\ref{cm-hq1}))

\begin{equation*}
Y_{t}=Y_{T}-\sum_{\alpha =1}^{m}\int_{t}^{T}Z^{\alpha }d\tilde{w}^{\alpha
}+\int_{t}^{T}\left[ \sum_{\alpha =1}^{m}Z^{\alpha }h_{z_{\alpha
}}(Y,Z)-h(Y,Z)\right] ds\text{,}
\end{equation*}%
thus%
\begin{equation*}
Y_{t}=Y_{T}-\sum_{\alpha =1}^{m}\int_{t}^{T}Z^{\alpha }d\tilde{w}^{\alpha }-%
\frac{1}{2}\int_{t}^{T}|Z_{s}|^{2}ds,
\end{equation*}%
and therefore%
\begin{eqnarray}
\mathbb{E}^{\mathbb{Q}}\left\{ \left. \frac{1}{2}\int_{t}^{T}|Z_{s}|^{2}ds%
\right\vert \mathcal{F}_{t}\right\} &=&\mathbb{E}^{\mathbb{Q}}\left\{ \left.
Y_{T}-Y_{t}\right\vert \mathcal{F}_{t}\right\}  \notag \\
&\leq &2||Y_{T}||_{\infty }\leq 2||\log u_{0}||_{\infty }\text{ .}
\label{bm-01}
\end{eqnarray}%
On the other hand, $M_{t}=e^{Kt}|Z_{t}|^{2}$ is a submartingale, thus one has%
\begin{equation*}
\mathbb{E}^{\mathbb{Q}}\left\{ \left. |Z_{s}|^{2}\right\vert \mathcal{F}%
_{t}\right\} \geq e^{K(t-s)}|Z_{t}|^{2},\text{ \ \ \ }\forall s\in \lbrack
t,T],
\end{equation*}%
so%
\begin{eqnarray}
\mathbb{E}^{\mathbb{Q}}\left\{ \left. \frac{1}{2}\int_{t}^{T}|Z_{s}|^{2}ds%
\right\vert \mathcal{F}_{t}\right\} &\geq &\frac{1}{2}%
\int_{t}^{T}e^{K(t-s)}|Z_{t}|^{2}ds  \notag \\
&=&\frac{1-e^{-K(T-t)}}{2K}|Z_{t}|^{2}\text{.}  \label{bm-02}
\end{eqnarray}%
Putting (\ref{bm-01}, \ref{bm-02}) together, we obtain%
\begin{equation*}
|Z_{t}|^{2}\leq \frac{4K}{1-e^{-K(T-t)}}||\log u_{0}||_{\infty },
\end{equation*}%
which yields (\ref{est-o1}).
\end{proof}

In general, we may proceed with $f=\psi (u)$ where $\psi $ is a concave
function, and $u$ is a positive solution to (\ref{he-hq1}), thus $f$
solves (\ref{eqf-hq1}) with
\begin{equation*}
h(y,z)=\frac{1}{2}\frac{\psi ^{^{\prime \prime }}(\psi ^{-1}(y))}{|\psi
^{\prime -1}(y))|^{2}}|z|^{2}\text{.}
\end{equation*}%
We can proceed as above. Under the probability $\mathbb{Q}$ 
\begin{equation*}
Y_{t}=Y_{T}-\sum_{\alpha =1}^{m}\int_{t}^{T}Z^{\alpha }d\tilde{w}^{\alpha
}+\int_{t}^{T}\left[ \sum_{\alpha =1}^{m}Z^{\alpha }h_{z_{\alpha
}}(Y,Z)-h(Y,Z)\right] ds\text{,}
\end{equation*}%
so%
\begin{equation*}
Y_{t}=Y_{T}-\sum_{\alpha =1}^{m}\int_{t}^{T}Z^{\alpha }d\tilde{w}^{\alpha }+%
\frac{1}{2}\int_{t}^{T}\frac{\psi ^{^{\prime \prime }}(\psi ^{-1}(Y_{s}))}{%
|\psi ^{\prime -1}(Y_{s}))|^{2}}|Z_{s}|^{2}ds.
\end{equation*}%
It is important to note that if 
\begin{equation*}
\psi ^{(3)}\psi ^{\prime }\leq 2|\psi ^{\prime \prime }|^{2},
\end{equation*}%
then $h_{y}(y,z)\geq 0$, thus from Lemma 2.2 in \cite{DelbaenHuBao},
\begin{equation}
\mathbb{E}^{\mathbb{Q}}\left\{ \left. \int_{t}^{T}|Z_{s}|^{2}ds\right\vert 
\mathcal{F}_{t}\right\} \leq 4||Y_{T}||_{\infty }^{2}\leq 4||\log
u_{0}||_{\infty }^{2}\text{ .}  \label{dhb1}
\end{equation}%
On the other hand $M_{t}=e^{Kt}|Z_{t}|^{2}$ is a submartingale, thus one has%
\begin{equation*}
\mathbb{E}^{\mathbb{Q}}\left\{ \left. |Z_{s}|^{2}\right\vert \mathcal{F}%
_{t}\right\} \geq e^{K(t-s)}|Z_{t}|^{2}, \text{ \ \ \ }\forall s\in \lbrack
t,T],
\end{equation*}%
and therefore%
\begin{eqnarray}
\mathbb{E}^{\mathbb{Q}}\left\{ \left. \int_{t}^{T}|Z_{s}|^{2}ds\right\vert 
\mathcal{F}_{t}\right\} &\geq &\int_{t}^{T}e^{K(t-s)}|Z_{t}|^{2}ds  \notag \\
&=&\frac{1-e^{-K(T-t)}}{K}|Z_{t}|^{2}\text{.}  \label{dhb2}
\end{eqnarray}%
Putting (\ref{dhb1}, \ref{dhb2}) together, we can obtain%
\begin{equation*}
|Z_{t}|^{2}\leq \frac{4K}{1-e^{-K(T-t)}}||\log u_{0}||_{\infty }^{2},
\end{equation*}%
which yields the following estimate.

\begin{theorem}
\label{ho-th2}Assume that Conditions (\ref{co-hq1}, \ref{co-hq2}, \ref%
{co-hq3}) are satisfied. Moreover, $\psi $ is concave and satisfies: 
\begin{equation*}
\psi ^{(3)}\psi ^{\prime }\leq 2|\psi ^{\prime \prime }|^{2}\text{ .}
\end{equation*}%
Then 
\begin{equation}
\sum_{\alpha =1}^{m}|A_{\alpha }\psi (u(t,x))|^{2}\leq \frac{4K}{%
1-e^{-K(T-t)}}||\psi (u_{0})||_{\infty }^{2}, \label{est-o2}
\end{equation}%
for any positive solution $u$ of (\ref{l-heat-hq1}).
\end{theorem}

\section{Heat equation on complete manifold}

In this section, we study positive solutions of the heat equation%
\begin{equation}
\left( \frac{1}{2}\Delta -\frac{\partial }{\partial t}\right) u=0,\text{ \ \
in }[0,\infty )\times M,  \label{he-hq31}
\end{equation}%
where $M$ is a complete manifold of dimension $n$, $\Delta $ is the
Beltrami-Laplace operator. In a local coordinate system so that the Riemann
metric $ds^{2}=g_{ij}dx^{i}dx^{j}$ and 
\begin{equation*}
\Delta =\frac{1}{\sqrt{g}}\sum_{i,j=1}^{n}\frac{\partial }{\partial x^{i}}%
g^{ij}\sqrt{g}\frac{\partial }{\partial x^{j}},
\end{equation*}%
where $g=\det (g_{ij})$ and $(g^{ij})$ is the inverse matrix of $(g_{ij})$.
We prove the following

\begin{theorem}
\label{th-man1}Suppose the Ricci curvature $Ric\geq -K$ for some $K\geq 0$,
and suppose $u$ is a positive solution of (\ref{he-hq31}) with initial data $%
u_{0}>0$. Then%
\begin{equation}
|\nabla \log u|^{2}(t,x)\leq \frac{2K}{1-e^{-\frac{Kt}{2}}}||\log u_{0}||_{\infty }%
\text{ .}  \label{gr-t1}
\end{equation}
\end{theorem}

\begin{remark}
This estimate can also be derived from the reverse logarithmic Sobolev inequality, as in
Remark \ref{remark:referee}.
\end{remark}

The preceding theorem is proved by using similar computations as in the
proof of Theorem 3.7 but working on the orthonormal frame bundle $O(M)$ over $M$.

Recall that a point $\gamma =(x,e)\in O(M)$, where $(e_{1},\cdots ,e_{n})$
is an orthonormal basis of the tangent space $T_{x}M$ at $x\in M$. Let $\pi
:\gamma =(x,e)\rightarrow x$ be the natural projection from $O(M)$ to $M$. $%
O(M)$ is a principal fibre bundle with its structure group $O(n)$. For the
general facts on differential geometry, we refer to Kobayashi and Nomizu 
\cite{MR1393940}.

Suppose $x=(x^{1},\cdots ,x^{n})$ is a local coordinate system on $M$, then
it induces a local coordinate system $\gamma =(x^{k},e_{j}^{i})$ on $O(M)$
so that $e_{j}=e_{j}^{i}\frac{\partial }{\partial x^{i}}$. If $L$ is a
vector field, then $\tilde{L}$ denotes the horizontal lifting of $L$ to $O(M)
$:%
\begin{equation*}
\tilde{L}(x,e)=L^{i}(x)\frac{\partial }{\partial x^{i}}-\Gamma
(x)_{ij}^{k}L^{i}(x)e_{l}^{j}\frac{\partial }{\partial e_{l}^{k}}
\end{equation*}%
in a local coordinate system, where $\Gamma _{ij}^{k}$ are the Christoffel
symbols associated with the Levi-Civita connection, and $L=L^{i}\frac{%
\partial }{\partial x^{i}}$. For $\alpha =1,\cdots ,n$ and $\gamma =(x,e)$,
then $\tilde{L}_{\alpha }$ denotes the horizontal lifting of $e_{\alpha }$,
that is%
\begin{equation*}
\tilde{L}_{\alpha }(x,e)=e_{\alpha }^{i}\frac{\partial }{\partial x^{i}}%
-\Gamma (x)_{ij}^{k}e_{\alpha }^{i}e_{l}^{j}\frac{\partial }{\partial
e_{l}^{k}}\text{.}
\end{equation*}%
The system $\{\tilde{L}_{1},\cdots ,\tilde{L}_{n}\}$ is called the system of
canonical horizontal vector fields. The mapping $\tilde{L}:\mathbb{R}%
^{n}\rightarrow \Gamma (TO(M))$ where $L_{\xi }=\xi ^{\alpha }\tilde{L}%
_{\alpha }$, is defined globally, and is independent of the choice of a
local coordinate system. Therefore 
\begin{equation*}
\Delta _{O(M)}=\sum_{\alpha =1}^{n}L_{\alpha }^{2}
\end{equation*}%
is well defined sub-elliptic operator of second order on the frame bundle $%
O(M)$, called the horizontal Laplacian. If $f\in C^{2}$ then $\Delta
_{O(M)}f\circ \pi =\Delta f$. For simplicity, any function $f$ on $M$ is
lifted to a function $\tilde{f}$ on $O(M)$ defined by $\tilde{f}=f\circ \pi $
which is invariant under the group action by $O(d)$.

The following relations will be used in what follows. 
\begin{equation}
\tilde{L}_{\alpha }\tilde{f}(\gamma )=\tilde{L}_{\alpha }f\circ \pi (\gamma
)=e_{\alpha }^{k}\frac{\partial f}{\partial x^{k}}\text{, \ \ for }\gamma
=(x^{i},e_{j}^{k})\in O(M)\text{ .}  \label{lf-lift1}
\end{equation}

We need the following geometric facts, whose proofs are elementary.

\begin{lemma}
\label{la-lem}For $\alpha =1,\cdots ,n$ we have 
\begin{equation}
\Delta _{O(M)}\tilde{L}_{\alpha }-\tilde{L}_{\alpha }\Delta _{O(M)}=\frac{1}{%
2}\sum_{\beta =1}^{n}[\tilde{L}_{\beta },\tilde{L}_{\alpha }]\tilde{L}%
_{\beta }+\frac{1}{2}\sum_{\beta =1}^{n}\tilde{L}_{\beta }[\tilde{L}_{\beta
},\tilde{L}_{\alpha }]\text{.}  \label{la-v1}
\end{equation}
\end{lemma}

\begin{lemma}
Suppose $f$ is a smooth function on $M$ and $\tilde{f}=f\circ \pi $ is the
horizontal lifting of $f$ to $O(M)$.

1) For $\alpha ,\beta =1,\cdots ,n$%
\begin{equation}
\lbrack \tilde{L}_{\alpha },\tilde{L}_{\beta }]\tilde{f}=0\text{ .}
\label{tor-1}
\end{equation}

2) We have%
\begin{equation}
\sum_{\alpha ,\beta =1}^{n}(\tilde{L}_{\alpha }\tilde{f})([\tilde{L}_{\beta
},\tilde{L}_{\alpha }]\tilde{L}_{\beta }\tilde{f})=\text{Ric}(\nabla
f,\nabla f),  \label{ric-1}
\end{equation}%
where Ric is the Ricci curvature.

3) We also have%
\begin{equation}
\Delta _{O(M)}\tilde{L}_{\alpha }\tilde{f}-\tilde{L}_{\alpha }\Delta _{O(M)}%
\tilde{f}=\frac{1}{2}\sum_{\beta =1}^{n}[\tilde{L}_{\beta },\tilde{L}%
_{\alpha }]\tilde{L}_{\beta }\tilde{f}\text{.}  \label{la-ric}
\end{equation}
\end{lemma}

\begin{proof}
The first identity (\ref{tor-1}) follows from the torsion-free condition. To
prove (\ref{ric-1}) we choose a local coordinate which is orthonormal and $%
dg_{ij}=0$ (so that $\Gamma _{ij}^{k}=0$) at the point we evaluate tensors,
thus%
\begin{equation*}
\tilde{L}_{\alpha }\tilde{L}_{\beta }\tilde{L}_{\beta }\tilde{f}=e_{\beta
}^{j}e_{\beta }^{i}e_{\alpha }^{q}\frac{\partial ^{3}f}{\partial
x^{q}\partial x^{i}\partial x^{j}}-\frac{\partial \Gamma _{ij}^{k}}{\partial
x^{q}}e_{\alpha }^{q}e_{\beta }^{i}e_{\beta }^{j}\frac{\partial f}{\partial
x^{k}}\text{ ,}
\end{equation*}%
\begin{equation*}
\tilde{L}_{\beta }\tilde{L}_{\alpha }\tilde{L}_{\beta }\tilde{f}=e_{\alpha
}^{j}e_{\beta }^{i}e_{\beta }^{q}\frac{\partial ^{3}f}{\partial
x^{q}\partial x^{i}\partial x^{j}}-e_{\beta }^{q}e_{\alpha }^{i}e_{\beta
}^{j}\frac{\partial \Gamma _{ij}^{k}}{\partial x^{q}}\frac{\partial f}{%
\partial x^{k}}\text{ ,}
\end{equation*}%
and therefore%
\begin{equation*}
\lbrack \tilde{L}_{\beta },\tilde{L}_{\alpha }]\tilde{L}_{\beta }\tilde{f}%
=\left( \frac{\partial \Gamma _{ij}^{k}}{\partial x^{q}}e_{\alpha
}^{q}e_{\beta }^{i}e_{\beta }^{j}-e_{\alpha }^{i}e_{\beta }^{q}e_{\beta }^{j}%
\frac{\partial \Gamma _{ij}^{k}}{\partial x^{q}}\right) \frac{\partial f}{%
\partial x^{k}},
\end{equation*}%
which yields (\ref{ric-1}). (\ref{la-ric}) comes from (4.4) and (4.5).
\end{proof}

Consider the following stochastic differential equation%
\begin{equation}
d\varphi (t)=\sum_{\alpha =1}^{n}\tilde{L}_{\alpha }(\varphi (t))\circ
dw_{t}^{\alpha }\text{, \ }\varphi (0)=\gamma \in O(M),  \label{om-01}
\end{equation}%
on the classical Wiener space $(\Omega ,\mathcal{F},\mathbb{P})$ of
dimension $n$. The stochastic flow associated with (\ref{om-01}) is denoted
by $\{\varphi (t,\cdot ,\gamma ):t\geq 0\}$, which is a diffusion process in 
$O(M)$ with the infinitesimal generator 
\begin{equation*}
\frac{1}{2}\Delta _{O(M)}=\frac{1}{2}\sum_{\alpha =1}^{n}\tilde{L}_{\alpha
}\circ \tilde{L}_{\alpha },
\end{equation*}%
in the sense that 
\begin{equation*}
M_{t}^{f}=f(t,\varphi (t))-f(0,\varphi (0))-\int_{0}^{t}\frac{1}{2}\Delta
_{O(M)}f(s,\varphi (s))ds
\end{equation*}%
is a local martingale for any $f\in C^{1,2}(\mathbb{R}_{+}\times O(M))$, and 
\begin{equation}
M_{t}^{f}=\sum_{\alpha =1}^{n}\int_{0}^{t}(L_{\alpha }f)(s,\varphi
(s))dw_{s}^{\alpha }\text{ .}  \label{ma-hq1}
\end{equation}%
We may express 
\begin{equation*}
\varphi (t,w,\gamma )=(X(t,w,\gamma ),E(t,w,\gamma )),
\end{equation*}%
where $X(t,w,\gamma )=\pi (\varphi (t,w,\gamma ))$, then $\{X(t,\cdot
,\gamma ):t\geq 0\}$ is a diffusion process on $M$ starting from $x=\pi
(\gamma )$ with infinitesimal generator $\frac{1}{2}\Delta $. $\{X(t,\cdot
,\gamma ):t\geq 0\}$ is a Brownian motion on $M$ starting from $x=\pi
(\gamma )$.

In a local coordinate system $(x^{k},e_{j}^{i})$, write 
\begin{equation*}
\varphi (t,\cdot ,\gamma )=(X^{k}(t,\cdot ,\gamma );E_{j}^{i}(t,\cdot
,\gamma ))
\end{equation*}%
so that $E_{\alpha }(t)=E_{\alpha }^{i}(t)\frac{\partial }{\partial x^{i}}$.
Then the SDE (\ref{om-01}) may be written as%
\begin{equation}
\left\{ 
\begin{array}{cc}
dX_{t}^{k}=\sum_{\alpha =1}^{n}E_{\alpha }^{k}(t)\circ dw_{t}^{\alpha }\text{%
, \ \ \ \ \ \ \ \ \ \ \ \ \ \ \ \ \ \ \ \ \ \ \ \ \ } &  \\ 
dE_{j}^{i}(t)=-\sum_{\alpha ,\beta ,k=1}^{n}\Gamma (X_{t})_{\beta
k}^{i}E_{j}^{k}(t)E_{\alpha }^{\beta }(t)\circ dw_{t}^{\alpha }\text{.} & 
\end{array}%
\right.   \label{h-b1}
\end{equation}

Let $F(t)=(F(t)_{j}^{i})=E(t)^{-1}$. Then $F(t)E(t)=I$ so that%
\begin{equation}
dF_{j}^{i}(t)=\sum_{\alpha ,\beta ,l=1}^{n}\Gamma (X_{t})_{\beta
j}^{l}F_{l}^{i}(t)E_{\alpha }^{\beta }(t)\circ dw_{t}^{\alpha }\text{.}
\label{h-b2}
\end{equation}

If $f\in C^{1,2}(\mathbb{R}_{+}\times M)$, then 
\begin{eqnarray}
f(T,X_{T}) &=&f(t,X_{t})+\int_{t}^{T}\sum_{\alpha ,k=1}^{n}E_{\alpha }^{k}(s)%
\frac{\partial f}{\partial x^{k}}(s,X_{s})dw_{s}^{\alpha }  \notag \\
&&+\int_{t}^{T}\left( \frac{\partial }{\partial s}+\frac{1}{2}\Delta \right)
f(s,X_{s})ds\text{,}  \label{ito-m11}
\end{eqnarray}%
so according to (\ref{lf-lift1})%
\begin{eqnarray}
\tilde{f}(T,X_{T}) &=&\tilde{f}(t,X_{t})+\int_{t}^{T}\sum_{\alpha =1}^{n}(%
\tilde{L}_{\alpha }\tilde{f})(s,(X_{s},E(s)))dw_{s}^{\alpha }  \notag \\
&&+\int_{t}^{T}\left( \frac{\partial }{\partial s}+\frac{1}{2}\Delta \right)
f(s,X_{s})ds\text{ .}  \label{ito-m1}
\end{eqnarray}

\subsection{Gradient estimate}

Suppose now $f$ satisfies the non-linear heat equation%
\begin{equation}
\left( \frac{1}{2}\Delta -\frac{\partial }{\partial t}\right) f=-\frac{1}{2}%
|\nabla f|^{2}\text{ .}  \label{f-t1}
\end{equation}%
Let $\tilde{f}$ be the horizontal lifting of $f$ i.e. $\tilde{f}=f\circ \pi $%
, so that $\tilde{f}$ satisfies the parabolic equation on $O(M)$:%
\begin{equation}
\left( \frac{1}{2}\Delta _{O(M)}-\frac{\partial }{\partial t}\right) \tilde{f%
}=-\frac{1}{2}\sum_{\beta =1}^{n}|\tilde{L}_{\beta }\tilde{f}|^{2}\text{ .}
\label{f-t2}
\end{equation}

Let $T>0$. Let $Y_{t}=f(T-t,X_{t})$ and 
\begin{equation}
Z_{t}^{\alpha }=\sum_{k=1}^{n}E_{\alpha }^{k}(t)\frac{\partial f}{\partial
x^{k}}(t,X_{t})=(\tilde{L}_{\alpha }\tilde{f})(t,(X_{t},E(t))),  \label{z-a1}
\end{equation}%
for $\alpha =1,\cdots ,n$. Then 
\begin{equation}
|\nabla f(T-t,\cdot )|^{2}(X_{t})=\sum_{\alpha =1}^{n}|Z_{t}^{\alpha }|^{2},
\label{da-hq1}
\end{equation}%
therefore, according to (\ref{ito-m11}),%
\begin{equation}
Y_{T}=Y_{t}+\int_{t}^{T}Z_{t}^{\alpha }dw_{s}^{\alpha }-\frac{1}{2}%
\int_{t}^{T}\sum_{\alpha =1}^{n}|Z_{t}^{\alpha }|^{2}ds\text{.}
\label{bs-hq21}
\end{equation}%
We will consider (\ref{bs-hq21}) as a backward stochastic differential
equation.

Applying It\^{o}'s formula to $\tilde{L}_{\alpha }\tilde{f}$ and the
stochastic flow $\varphi (t,\cdot ,\gamma )$ to obtain 
\begin{eqnarray}
dZ_{t}^{\alpha } &=&\sum_{\beta =1}^{n}\tilde{L}_{\beta }(\tilde{L}_{\alpha }%
\tilde{f})(t,\varphi (t,\cdot ,\gamma ))dw_{t}^{\beta }  \notag \\
&&+\left( \frac{1}{2}\Delta _{O(M)}+\frac{\partial }{\partial t}\right) 
\tilde{L}_{\alpha }\tilde{f}(t,\varphi (t,\cdot ,\gamma ))dt\text{,}
\label{z-h31}
\end{eqnarray}%
so%
\begin{eqnarray}
d|Z|^{2} &=&2\sum_{\alpha =1}^{n}Z_{t}^{\alpha }dZ_{t}^{\alpha
}+\sum_{\alpha ,\beta =1}^{n}|\tilde{L}_{\beta }\tilde{L}_{\alpha }\tilde{f}%
|^{2}(t,\varphi (t,\cdot ,\gamma ))dt  \notag \\
&=&2\sum_{\alpha ,\beta =1}^{n}Z_{t}^{\alpha }(\tilde{L}_{\beta }\tilde{L}%
_{\alpha }\tilde{f})dw_{t}^{\beta }+\sum_{\alpha ,\beta =1}^{n}|\tilde{L}%
_{\beta }\tilde{L}_{\alpha }\tilde{f}|^{2}dt  \notag \\
&&+2\sum_{\alpha =1}^{n}Z_{t}^{\alpha }\left( \frac{1}{2}\Delta _{O(M)}+%
\frac{\partial }{\partial t}\right) \tilde{L}_{\alpha }\tilde{f}dt\text{ .}
\label{z-h32}
\end{eqnarray}

However, by Lemma 4.3, 
\begin{eqnarray*}
\left( \frac{1}{2}\Delta _{O(M)}+\frac{\partial }{\partial t}\right) \tilde{L%
}_{\alpha }\tilde{f} &=&\tilde{L}_{\alpha }\left( \frac{1}{2}\Delta _{O(M)}+%
\frac{\partial }{\partial t}\right) \tilde{f}+\frac{1}{4}\sum_{\beta =1}^{n}[%
\tilde{L}_{\beta },\tilde{L}_{\alpha }]\tilde{L}_{\beta }\tilde{f} \\
&=&-\sum_{\beta =1}^{n}(\tilde{L}_{\beta }\tilde{f})(\tilde{L}_{\alpha }%
\tilde{L}_{\beta }\tilde{f})+\frac{1}{4}\sum_{\beta =1}^{n}[\tilde{L}_{\beta
},\tilde{L}_{\alpha }]\tilde{L}_{\beta }\tilde{f}\text{,}
\end{eqnarray*}%
we therefore have%
\begin{eqnarray}
d|Z|^{2} &=&2\sum_{\alpha =1}^{n}Z_{t}^{\alpha }\sum_{\beta =1}^{n}(\tilde{L}%
_{\beta }\tilde{L}_{\alpha }\tilde{f})\left( dw_{t}^{\beta }-Z_{t}^{\beta
}dt\right)  \notag \\
&&+\sum_{\alpha ,\beta =1}^{n}|\tilde{L}_{\beta }\tilde{L}_{\alpha }\tilde{f}%
|^{2}dt+\frac{1}{2}\sum_{\alpha ,\beta =1}^{n}(\tilde{L}_{\alpha }\tilde{f}%
)([\tilde{L}_{\beta },\tilde{L}_{\alpha }]\tilde{L}_{\beta }\tilde{f})dt%
\text{.}  \notag
\end{eqnarray}

By using Lemma 4.3 we obtain%
\begin{eqnarray}
d|Z|^{2} &=&2\sum_{\alpha =1}^{n}Z_{t}^{\alpha }\sum_{\beta =1}^{n}(\tilde{L}%
_{\beta }\tilde{L}_{\alpha }\tilde{f})\left( dw_{t}^{\beta }-Z_{t}^{\beta
}dt\right)  \notag \\
&&+\sum_{\alpha ,\beta =1}^{n}|\tilde{L}_{\beta }\tilde{L}_{\alpha }\tilde{f}%
|^{2}dt+\frac{1}{2}\text{Ric}(\nabla f,\nabla f)(T-t,X_{t})dt\text{.}
\label{key-01}
\end{eqnarray}

Define a probability $\mathbb{Q}$ by $\left. \frac{d\mathbb{Q}}{d\mathbb{P}}%
\right\vert _{\mathcal{F}_{T}}=R_{T}$ where 
\begin{equation*}
R_{t}=\exp \left[ \int_{0}^{t}Z^{\beta }dw^{\beta }-\frac{1}{2}%
\int_{0}^{t}\left\vert Z\right\vert ^{2}ds\right] \text{.}
\end{equation*}%
Then 
\begin{equation*}
Y_{T}=Y_{t}+\int_{t}^{T}Z_{t}^{\alpha }d\tilde{w}_{s}^{\alpha }+\frac{1}{2}%
\int_{t}^{T}\sum_{\alpha =1}^{n}|Z_{t}^{\alpha }|^{2}ds\text{,}
\end{equation*}%
hence%
\begin{equation}
\mathbb{E}^{\mathbb{Q}}\left\{ \left. \int_{t}^{T}|Z_{s}|^{2}ds\right\vert 
\mathcal{F}_{t}\right\} \leq 4||f_{0}||_{\infty }\text{.}  \label{q-bmo}
\end{equation}

But on the other hand $e^{\frac{Kt}{2}}|Z_{t}|^{2}$ is submartingale under $\mathbb{Q}$%
, thus%
\begin{eqnarray*}
\mathbb{E}^{\mathbb{Q}}\left\{ \left. \int_{t}^{T}|Z_{s}|^{2}ds\right\vert 
\mathcal{F}_{t}\right\} &=&\int_{t}^{T}\mathbb{E}^{\mathbb{Q}}\left\{ \left.
|Z_{s}|^{2}\right\vert \mathcal{F}_{t}\right\} ds \\
&\geq &e^{\frac{Kt}{2}}|Z_{t}|^{2}\int_{t}^{T}e^{-\frac{Ks}{2}}ds \\
&=&\frac{1-e^{-\frac{K}{2}(T-t)}}{\frac{K}{2}}|Z_{t}|^{2},
\end{eqnarray*}%
which yields that%
\begin{equation*}
|Z_{t}|^{2}\leq \frac{2K}{1-e^{-\frac{K}{2}(T-t)}}||f_{0}||_{\infty },
\end{equation*}%
and hence (\ref{gr-t1}).

\section{Li-Yau's estimate}

In this section we prove Theorem \ref{in-th2}. Thus, $u$ is a positive
solution to the heat equation 
\begin{equation}
\left( \frac{1}{2}\Delta -\frac{\partial }{\partial t}\right) u=0,\text{ \ \
\ on }[0,\infty )\times M,  \label{heat-d}
\end{equation}%
where $M$ is a complete Riemannian manifold of dimension $n$, with
non-negative Ricci curvature. Then $f=\log u$ is a solution to the
semi-linear heat equation%
\begin{equation}
\left( \frac{1}{2}\Delta -\frac{\partial }{\partial t}\right) f=-\frac{1}{2}%
|\nabla f|^{2}\text{, \ \ }f(0,\cdot )=f_{0}\text{.}  \label{li-hq1}
\end{equation}

Taking derivative in the parabolic equation (\ref{li-hq1}) with respect to
the time parameter $t$ one obtains%
\begin{equation}
\frac{\partial }{\partial t}f_{t}=Lf_{t},  \label{li-hq2}
\end{equation}%
where%
\begin{equation*}
L=\frac{1}{2}\Delta +\nabla f.\nabla \text{ .}
\end{equation*}%
By using the Bochner identity one can verify that%
\begin{equation}
\frac{\partial }{\partial t}|\nabla f|^{2}=L|\nabla f|^{2}-|\nabla \nabla
f|^{2}-2\text{Ric}(\nabla f,\nabla f)\text{ .}  \label{li-hq3}
\end{equation}

Since $\Delta f$ is the trace of the Hessian $\nabla \nabla f$ so that $%
|\nabla \nabla f|^{2}\geq \frac{1}{n}\left( \Delta f\right) ^{2}$, thus,
since the Ricci curvature is non-negative, then%
\begin{equation}
\frac{\partial }{\partial t}|\nabla f|^{2}\leq L|\nabla f|^{2}-\frac{1}{n}%
\left( \Delta f\right) ^{2}\text{ .}  \label{li-hq4}
\end{equation}%
Let 
\begin{equation*}
G=-\Delta f=|\nabla f|^{2}-2f_{t}\text{ .}
\end{equation*}%
By combining (\ref{li-hq2}) and (\ref{li-hq4}) together, we obtain%
\begin{equation}
\frac{\partial }{\partial t}G\leq LG-\frac{1}{n}G^{2}\text{ .}
\label{li-hq5}
\end{equation}

Suppose that $C>0$ and $-\nabla\log u_0\le C$.  Set
\begin{equation*}
F=\left( \frac{t}{n}+\frac{1}{C}\right) \left( -\Delta f\right).
\end{equation*}%
Since%
\begin{eqnarray*}
\left( \frac{\partial }{\partial t}-L\right) F &=&\frac{1}{n}G+\left( \frac{t%
}{n}+\frac{1}{C}\right) \left( \frac{\partial }{\partial t}-L\right) G \\
&\leq &\frac{1}{n}G-\left( \frac{t}{n}+\frac{1}{C}\right) \frac{1}{n}G^{2} \\
&=&\frac{1}{n}G\left( 1-F\right) \text{ ,}
\end{eqnarray*}%
and therefore%
\begin{equation*}
\left( L-\frac{\partial }{\partial t}\right) (F-1)\geq \frac{1}{n}G\left(
F-1\right) \text{ .}
\end{equation*}%

As $-\Delta\log u_0\le C$, $(F-1)(0,\cdot)\le 0$.
Applying the maximum principle, $F-1\le 0$, from which 
we obtain the  estimate in Theorem \ref{in-th2}. 

\begin{corollary}
Suppose $u$ is a positive solution to (\ref{heat-d}), then the following
Harnack inequality holds:%
\begin{equation*}
\frac{u(t,x)}{u(t+s,y)}\leq \exp \left[ \int_{t}^{t+s}K(\sigma ,\frac{\rho }{%
s})d\sigma \right] 
\end{equation*}%
where 
\begin{equation*}
K(t,\alpha )=\sup_{Y:\psi (t,Y)\geq 0}\left\{ \alpha \sqrt{\psi (t,Y)}%
-Y\right\} 
\end{equation*}%
and 
\begin{equation*}
\psi (t,Y)=2Y+\frac{C}{\frac{t}{n}C+1}\text{ .}
\end{equation*}%
Thus%
\begin{equation*}
\frac{u(t,x)}{u(t+s,y)}\leq \left( \frac{\frac{1}{C}+\frac{1}{n}(t+s)}{\frac{%
1}{C}+\frac{1}{n}t}\right) ^{\frac{n}{2}}\exp \left[ \frac{r(x,y)^{2}}{2s}%
\right] \text{ .}
\end{equation*}
\end{corollary}

This follows by integrating the gradient estimates in Theorem \ref{in-th2}
along geodesics, see \cite{MR1681640} for details.

\medskip

{\bf Acknowledgement} The authors would like to thank the referee for the
helpful comments.

\end{document}